\documentclass{amsart}
\usepackage{amssymb}
\usepackage{amsfonts}
\usepackage{amssymb}
\usepackage{amsmath}
\usepackage{amsthm}
\usepackage{enumerate}
\usepackage{tabularx}
\usepackage{centernot}
\usepackage{mathtools}
\usepackage{stmaryrd}
\usepackage{amsthm,amssymb}
\usepackage{etoolbox}
\usepackage{tikz}
\usepackage{amssymb}
\usepackage{amssymb,
enumitem,
hyperref,
tikz,
upgreek,
mathtools,
verbatim,
hyphenat}
\hypersetup{colorlinks=false}
\numberwithin{equation}{section}
\usepackage{marginnote}
\definecolor{mygray}{gray}{0.85}
\usepackage[linecolor=black,backgroundcolor=mygray,colorinlistoftodos,prependcaption,textsize=small]{todonotes}
\usepackage{xargs}   
\usepackage[english]{babel}

\newtheorem{theorem}{Theorem}

\newtheorem{corollary}[theorem]{Corollary}
\newtheorem{fact}[theorem]{Fact}
\newtheorem{proposition}[theorem]{Proposition}
\newtheorem{question}[theorem]{Question}

\theoremstyle{definition}
\newtheorem{definition}[theorem]{Definition}
\newtheorem{notation}[theorem]{Notation}
\newtheorem{convention}[theorem]{Convention}

\theoremstyle{remark}
\newtheorem{remark}[theorem]{Remark}

\newcommand{\NN}{\mathbb{N}}

\newcommand{\XX}{\mathbb{X}}

\newcommand{\setm}[2]{\{ #1 \mid #2 \}}

\newcommand{\embeds}{\sqsubseteq}
\newcommand{\analytic}{\boldsymbol{\Sigma}_1^1}

\newcommand{\Aut}[1]{\mathrm{Aut}(#1)}

\newcommand{\Subg}[1]{\mathrm{Subg}(#1)}

\newenvironment{enumerate-(a)}{\begin{enumerate}[label={\upshape (\alph*)}, leftmargin=2pc]}{\end{enumerate}}
\newenvironment{enumerate-(a)-r}{\begin{enumerate}[label={\upshape (\alph*)}, leftmargin=2pc,resume]}{\end{enumerate}}
\newenvironment{enumerate-(a)-5}{\begin{enumerate}[label={\upshape (\alph*)}, leftmargin=2pc,start=5]}{\end{enumerate}}
\newenvironment{enumerate-(A)}{\begin{enumerate}[label={\upshape (\Alph*)}, leftmargin=2pc]}{\end{enumerate}}
\newenvironment{enumerate-(A')}{\begin{enumerate}[label={\upshape (\Alph*')}, leftmargin=2pc]}{\end{enumerate}}
\newenvironment{enumerate-(B')}{\begin{enumerate}[label={\upshape (B')}, leftmargin=2pc]}{\end{enumerate}}
\newenvironment{enumerate-(A)-r}{\begin{enumerate}[label={\upshape (\Alph*)}, leftmargin=2pc,resume]}{\end{enumerate}}
\newenvironment{enumerate-(i)}{\begin{enumerate}[label={\upshape (\roman*)}, leftmargin=2pc]}{\end{enumerate}}
\newenvironment{enumerate-(i)-r}{\begin{enumerate}[label={\upshape (\roman*)}, leftmargin=2pc,resume]}{\end{enumerate}}
\newenvironment{enumerate-(I)}{\begin{enumerate}[label={\upshape (\Roman*)}, leftmargin=2pc]}{\end{enumerate}}
\newenvironment{enumerate-(I)-r}{\begin{enumerate}[label={\upshape (\Roman*)}, leftmargin=2pc,resume]}{\end{enumerate}}
\newenvironment{enumerate-(1)}{\begin{enumerate}[label={\upshape (\arabic*)}, leftmargin=2pc]}{\end{enumerate}}
\newenvironment{enumerate-(1)-r}{\begin{enumerate}[label={\upshape (\arabic*)}, leftmargin=2pc,resume]}{\end{enumerate}}
\newenvironment{itemizenew}{\begin{itemize}[leftmargin=2pc]}{\end{itemize}}

\makeatletter
\def\subsection{\@startsection{subsection}{3}%
  \z@{.5\linespacing\@plus.7\linespacing}{.3\linespacing}%
  {\bfseries\centering}}
\makeatother

\makeatletter
\def\subsubsection{\@startsection{subsubsection}{3}%
  \z@{.5\linespacing\@plus.7\linespacing}{.3\linespacing}%
  {\centering}}
\makeatother

\begin{document}

\title{Invariant Universality for Projective Planes}
\date{\today}
\author[F.~Calderoni]{Filippo Calderoni}

\address{Dipartimento di matematica ``Giuseppe Peano'', Universit\`a di Torino, Via Carlo Alberto 10, 10121 Torino --- Italy}
\email{filippo.calderoni@unito.it}

\author[G.~Paolini]{Gianluca Paolini}

\address{Einstein Institute of Mathematics,  The Hebrew University of Jerusalem, Israel}
\email{gianluca.paolini@mail.huji.ac.il}

\thanks{The first author thanks Andrew Brooke-Taylor for pointing out~\cite{PT} and interesting discussions.}

\begin{abstract} 

We continue the work of \cite{BroCalMil,CalMot,CamMarMot} by analyzing the equivalence relation of bi-embeddability on various classes of countable planes, most notably the class of countable non-Desarguesian projective planes. We use constructions of the second author to show that these equivalence relations are invariantly universal, in the sense of \cite{CamMarMot}, and thus in particular complete analytic. We also introduce a new kind of Borel reducibility relation for standard Borel \(G\)-spaces, which requires the preservation of stabilizers, and explain its connection with the notion of full
\mbox{embeddings
commonly considered in category theory.}
\end{abstract}

\maketitle

\section{Introduction}

	\begin{definition}\label{def_plane} A plane is a system of points and lines satisfying:
	\begin{enumerate-(A)}
	\item every pair of distinct points determines a unique line;
	\item every pair of distinct lines intersects in at most one point;
	\item every line contains at least two points;
	\item there exist at least three non-collinear points.
\end{enumerate-(A)}
A plane is {\em projective} if in addition:
	\begin{enumerate-(B')}
	\item every pair of lines intersects in exactly one point.
\end{enumerate-(B')}
A plane is {\em simple} if except for a finite number of points every point is incident with at most two non-trivial lines (i.e. lines containing more than two points).
\end{definition}

	The class of simple planes and the class of (non-Desarguesian) projective planes are first-order classes, and so we can regard them as standard Borel spaces, and use invariant descriptive set theory to analyze the complexity of analytic equivalence relations defined on them.
We recall that a binary relation \(R\) defined on a standard Borel space \(X\) is called \emph{analytic} (or \(\analytic\)), if it is an analytic subset of the product space \(X\times X\), i.e., it is the projection of a Borel set \(B\subseteq Y\times X\times X\), for some Polish space \(Y\).
	
The main tool to compare equivalence relations is what is known as {\em Borel reducibility}. If \(E\) and \(F\) are two equivalence relations on the standard Borel spaces \(X\) and \(Y\), we say that \(E\) \emph{Borel reduces} to \(F\) (and write \(E \leq_\mathrm{B} F\)) if there is a Borel map \(f \colon X \to Y\) witnessing that \(x\mathbin{E}y \iff f(x)\mathbin{F}f(y)\), for every \(x,y\in X\). 
We can take the statement ``\(E\) Borel reduces to \(F\)'' as a formal way of saying that 
\(E\) is not more complicated than \(F\), as any set of complete invariants for \(F\) includes a set of complete invariants for \(E\). When \(E\leq_\mathrm{B} F\) and \(F\leq_{B } E\), the complexity of \(E\) and \(F\) is considered the same, and we say that \(E\) and \(F\) are \emph{Borel bi-reducible} (in symbols, \(E\sim_{B} F\)).

In \cite{LouRos} the authors proved that the bi-embeddability relation \(\equiv_\mathrm{Gr}\) on countable graphs is a \emph{complete analytic equivalence relation}. That is, \(\equiv_\mathrm{Gr}\) is a \(\leq_\mathrm{B}\)-maximum among all analytic equivalence relations. 
It follows that \(\equiv_\mathrm{Gr}\) is strictly more complicated than any isomorphism relation between countable structures, and so it can be argued that the problem of classifying countable graphs up to bi-embeddability is highly intractable.

In \cite{FriMot} the authors proved that the bi-embeddability relation on countable graphs is analytic complete in a very strong sense: every analytic equivalence relation is Borel bi-reducible with the restriction of \(\equiv_\mathrm{Gr}\) to some \(\mathfrak{L}_{\omega_{1}\omega}\)-subclass of the standard Borel space of countable graphs. Such property reappeared thereafter in~\cite{CamMarMot}, where it was considered in a more general framework and
called \emph{invariant universality} --- the definition given in \cite{CamMarMot} is stated for all analytic equivalence relations (not only for those defined on spaces of countable structures).

Next, the work of \cite{CamMarMot} was continued by the first author of this paper et al., who proved invariant universality for the bi-embeddibility relation on several \(\mathfrak{L}_{\omega_{1}\omega}\)-classes, which include countable groups (cf. \cite[Theorem~3.5]{CalMot}), and countable fields of fixed characteristic \(p\neq 2\) (cf. \cite[Theorem~5.12]{BroCalMil}). The main technique used in \cite{BroCalMil} and \cite{CalMot} require to have a Borel reduction from the bi-embeddability relation between graphs to the bi-embeddability relation on the class under consideration, and the possibility to explicitly describe
the automorphism group of each structure in the image of the reduction.

In \cite{Pao_proj} the second author proved the Borel completeness
of both the class of simple planes and the class of non-Desarguesian projective planes. That is, the isomorphism relation on both of those classes of planes is a \(\leq_\mathrm{B}\)-maximum for all orbit equivalence relations arising from a Borel action of \(S_{\infty}\), the Polish group of permutations on \(\NN\).
In each case he defined a Borel reduction from the isomorphism relation between countable graphs to the isomorphism relation between the class under consideration. Furthermore, his constructions have the remarkable additional property of preserving automorphism groups. As we point down in the last section this feature is common to many categorical construction which give a full embedding between two \(\mathfrak{L}_{\omega_{1}\omega}\)-class, and can be adapted to define a Borel reduction between the isomorphism relations defined on the corresponding standard Borel spaces.

Our aim in this paper is twofold:
\begin{itemize}
	\item To study the bi-embeddability relation on the classes of countable planes previously considered in \cite{Pao_proj},
with the stipulation that the bi-embeddability relation between planes coincide with the bi-embeddability relation between the corresponding geometric lattices.
	\item To develop some generalities on the kind of stabilizer preserving Borel reduction (or SPB reduction for short) mentioned above.
\end{itemize} 
Concerning the first aim, we use the main constructions of \cite{Pao_proj} to prove:

\begin{theorem}\label{th_inva_un1}
The bi-embeddability relation \(\equiv_\mathrm{pl}\) between countable simple planes is invariantly universal.
\end{theorem}

\begin{theorem}\label{th_inva_un2}
The bi-embeddability relation \(\equiv_\mathrm{ppl}\) between countable non-{Desar-guesian} projective planes is invariantly universal.
\end{theorem}

\begin{corollary}\label{first_cor}
Every \(\analytic\) equivalence relation is Borel bi-reducible with the bi\hyp{}embeddability relation restricted to some \(\mathfrak{L}_{\omega_{1}\omega}\)-subclass of countable simple planes.
\end{corollary}

\begin{corollary}\label{second_cor}
Every \(\analytic\) equivalence relation is Borel bi-reducible with the bi\hyp{}embeddability relation restricted to some \(\mathfrak{L}_{\omega_{1}\omega}\)-subclass of countable non-Desarguesian projective planes.
\end{corollary}

Consequently, the bi-embeddability relation in the class of~countable
non\hyp{}Desarguesian projective planes is strictly more complicated than isomorphism.
In fact, we get that \(\equiv_\mathrm{ppl}\) is a complete analytic equivalence relation in the sense of~\cite[Definition~1.2]{LouRos}. It follows that we cannot classify the class of countable non\hyp{}Desarguesian projective planes up to bi-embeddability in any reasonable way: neither in terms of Ulm-type invariants, nor in terms of orbits of Polish group actions.

Concerning the second aim, we point out how in some cases SPB reductions can be obtained from the existing literature in category theory and list a couple of open questions.

\section{Invariant Universality}

Following \cite{LouRos} we consider Borel reducibility between quasi-orders, i.e., reflexive and transitive binary relations.

\begin{definition}
Let \(Q\) and \(R\) be quasi-orders on the standard Borel spaces \(X\) and \(Y\).
\begin{itemizenew}

\item \(Q\) \emph{Borel reduces} to \(R\) (in symbols, \(Q\leq_\mathrm{B} R\)) if there exists a Borel map
\(f\colon X\to Y\) such that for all \(x,y\in X\), \[x\mathbin{Q}y\quad\iff\quad{f(x)}\mathbin{R}{f(y)}\, .\]In this case we say that \(f\) is a Borel reduction from \(Q\) to \(R\) (in symbols, \({f\colon Q\leq_\mathrm{B}R}\)).
\item \(Q\) is \emph{Borel bi-reducible} with \(R\) (in symbols, \(Q\sim_{B} R\)) if \(Q\leq_\mathrm{B}R\) and \(R\leq_\mathrm{B}Q\).
\end{itemizenew}
\end{definition}

 In particular, when \(Q\) and \(R\) are equivalence relations, one obtains the usual notion of Borel reducibility previously mentioned in the introduction.
When \(Q\) is an analytic quasi-order on \(X\) and \(A\) is a Borel subset of \(X\), we can regard \(A\) as a standard Borel space with its relative standard Borel structure and the quasi-order on \(A\) obtained by the restriction of \(Q\). We shall denote by \(Q\restriction A\) the restriction of \(Q\) over \(A\).

We now recall the main definitions from \cite[Definition~1.1]{CamMarMot}.

 \begin{definition}\label{Definition : invariantly universal}
Let \( Q \) be a \( \analytic \) quasi-order on some standard Borel space \( X \) and let \(E\) be a \( \analytic \) equivalence subrelation of \( Q \). We say that \( (Q,E) \)
is \emph{invariantly universal} if for every \( \analytic \) quasi-order \( P \) there is a Borel subset \( A\subseteq X \) which is \(E\)-invariant and
such that  \( P \sim_{B} Q \restriction A \).
 \end{definition}
 \begin{definition}\label{Definition : invariantly universalER}
 Let \( F \) be a \( \analytic \) equivalence relation on some standard Borel space \( X \) and let \(E\) be a \( \analytic \) equivalence subrelation of \( F \). We say that \( (F,E) \)
is \emph{invariantly universal} if for every \( \analytic \) equivalence relation \( D \) there is a Borel subset \( A\subseteq X \) which is \(E\)-invariant and
such that  \( D  \sim_{B} F \restriction A \).
 \end{definition}
 
Notice that if \((F,E)\) is invariantly universal, then \(F\) is in particular a complete analytic equivalence relation in the sense of \cite[Definition~1.2]{LouRos}. Moreover, our interest for quasi-orders is easily explained: if \((Q, E)\) is an invariantly universal quasi-order and \(E_{Q}\) is the equivalence relation generated by \(Q\), then \((E_{Q} , E)\) is
an invariantly universal equivalence relation.

Throughout this paper we will make use of the following notation.

	\begin{notation}\label{notation: relations} Let \(X\) be a standard Borel space of countable structures.
	\begin{enumerate-(i)}
	\item We denote by \(\sqsubseteq_X\) (or simply $\sqsubseteq$) the embeddability relation on \( X \). 
	\item We denote by \(\cong_X \) (or simply $\cong$) the isomorphism relation on \( X \).
	\item We denote by \(\equiv_X \) (or simply $\equiv$) the bi-embeddability relation on \( X \).
	\item We say that the quasi-order \(Q\) on \(X\) is invariantly universal if \({(Q,\cong_X)}\) is (cf.~Definition~\ref{Definition : invariantly universal}).
	\item We say that the equivalence relation \(E\) on \(X\) is invariantly universal if \({(E,\cong_X)}\) is (cf.~Definition~\ref{Definition : invariantly universalER}).
	\end{enumerate-(i)}
\end{notation}

The following fact is an immediate consequence of L\'opez-Escobar theorem (cf.~\cite[Theorem~16.8]{Kec}) and gives a further insight of the phenomenon of invariantly universality on spaces of countable structures.

\begin{fact}\label{fact: elmentary subclasses}
If \(X\) is a standard Borel space of countable structures, and \(F\) is a \(\analytic\) equivalence relation on \(X\), then \(F\) is invariantly universal if and only if every \(\analytic\) equivalence relation is Borel bi-reducible with the restriction of \(F\) to some \(\mathfrak{L}_{\omega_{1}\omega}\) subclass of \(X\).
\end{fact}

We now present a sufficient condition for invariant universality. Let \(X_\mathrm{Gr}\) be the standard Borel space of countable graphs. First we abstract the following fact from~\cite[Section~3]{CamMarMot}. 
	\begin{fact}\label{proposition : FriMot}
There is a Borel subset \(\XX \subseteq X_\mathrm{Gr}\) such that the following hold:
\begin{enumerate-(i)}
\item \label{property1} the equality and isomorphism relations restricted to \( \\X \), denoted respectively by \( =_{\XX} \) and \( \cong_{\XX} \), coincide;
\item \label{property2} each graph in \( \XX \) is rigid; that is,  it has no non-trivial automorphism;
\item for every \(\analytic\) quasi-order \(P\) on \( 2^{\NN}\), there exists an injective Borel reduction \( \alpha\mapsto T_{\alpha} \) from \(P\) to \(\embeds_{\XX}\).
\end{enumerate-(i)}
\end{fact}

\begin{notation}
We denote by \(S_{\infty}\) the Polish group of permutations on \(\NN\), and
by \(\Subg{S_{\infty}}\) the
standard Borel space of closed subgroup of \(S_{\infty}\), endowed with the Effros-Borel structure (see~\cite[Section~12.C]{Kec}).
\end{notation}

Now we recall the following fact, which is a particular case of \cite[Theorem~4.2]{CamMarMot}.

\begin{fact}\label{theorem : suff condition} Let \(X\) be a standard Borel space of countable structures. Then the relation \(\embeds_{X}\) is an invariantly universal quasi-order
provided that the following conditions hold:
\begin{enumerate-(I)}\label{enumerate : TheoremCMMR13}
 \item \label{conditionI : CMMR13} there is a Borel map \(f \colon \mathbb{X}\rightarrow X\) such that:
 \begin{enumerate-(i)}
 \item \label{i : CMMR13}\(f\colon {\embeds_\mathbb{X}}\leq_\mathrm{B}{\embeds_{X}}\);
 \item \label{ii : CMMR13}\(f\colon {\cong_\mathbb{X}}\leq_\mathrm{B}{\cong_{X}}\);
  \end{enumerate-(i)}
\item \label{conditionII : CMMR13} the map
\(
f(\mathbb{X})\to\Subg { S_{\infty}}: f(T)\mapsto \Aut{f(T)}\)
is Borel.
\end{enumerate-(I)}
\end{fact}

We stress:

\begin{remark}\label{remark: simplifiedsuffcondition}
Since every graph in \(\mathbb{X}\) is rigid, whenever the reduction \(f\) witnessing~\ref{conditionI : CMMR13} of Fact~\ref{theorem : suff condition} further preserves the automorphisms groups, condition~\ref{conditionII : CMMR13} is automatically satisfied.
\end{remark}

\section{Planes}

	\begin{definition}\label{def_subplane} Let $P_1$ and $P_2$ be planes (cf. Definition~\ref{def_plane}).
	\begin{enumerate-(1)}
	\item We say that $P_1$ is a {\em subplane} of $P_2$ if $P_1 \subseteq P_2$, points of $P_1$ are points of $P_2$, lines of $P_1$ are lines of $P_2$, and the point $p$ is on the line $\ell$ in $P_1$ if and only if the point $p$ is on the line $\ell$ in $P_2$.
	\item We say that $P_1$ is a {\em complete subplane} of $P_2$ is $P_1$ is a subplane of $P_2$ and any point of intersection of lines of $P_1$ which lies in $P_2$ also lies in $P_1$, and every line joining two points of $P_1$ which lies in $P_2$ also lies in $P_1$.
\end{enumerate-(1)}	
\end{definition}

\begin{definition}[Cf. {\cite[Theorem 11.4]{piper}}]\label{def_free_ext} Given a plane $P$ we define by induction on $n < \omega$ a chain of planes $(P_n : n < \omega)$ as follows:
\newline $n = 0$. Let $P_n = P$.
\newline $n = m+1$. For every pair of parallel lines $\ell \neq \ell'$ in $P_m$ add a new point $\ell \wedge \ell'$ to $P_m$ incident with only $\ell$ and $\ell'$. Let $P_n$ be the resulting plane.
\newline We define the {\em free projective extension} of $P$ to be $F(P) \coloneqq\bigcup_{n < \omega} P_n$.
\end{definition}

	\begin{definition}\label{def_conf} Let $P$ be a plane.
	\begin{enumerate-(1)}
	\item We say that a line from $P$ is {\em trivial} if it contains exactly two points from $P$.
	\item If $P$ is {\em finite}, then we say that $P$ is {\em confined} if every point of $P$ is incident with at least three lines of $P$, and every line of $P$ is non-trivial.
	\item We say that $P$ is confined if every point and every line of $P$ is contained in a finite confined subplane of $P$.
\end{enumerate-(1)}
\end{definition}

	\begin{fact}\label{fact_proj1} Let $P_1$ and $P_2$ be confined planes (cf. Definition~\ref{def_conf}) and $f\colon P_1 \rightarrow P_2$ a complete embedding (i.e. $f(P_1)$ is a complete subplane of $P_2$). Then there exists a complete embedding $\hat{f} \colon F(P_1) \rightarrow F(P_2)$ such that $f \restriction P_1 = f$.
\end{fact}

	\begin{proof} This is essentially \cite[Theorem 4.3]{hall_proj}.
\end{proof}

	\begin{fact}\label{fact_proj2}  Let $P_1$ and $P_2$ be confined planes (cf. Definition~\ref{def_conf}) and $f\colon F(P_1) \rightarrow F(P_2)$ an embedding. Then $f \restriction P_1 \subseteq P_2$.
\end{fact}

	\begin{proof} In the terminology of \cite[Chapter XI]{piper}, the core of $F(P_{\ell})$ equals the core of $P_\ell$ (cf. \cite[Corollary at p. 224]{piper}), which in turn is the whole of $P_\ell$, since by assumption $P_\ell$ is confined. It easily follows that $f \restriction P_1 \subseteq P_2$, since otherwise the core of $F(P_2)$ is not equal to $P_2$.
\end{proof}

	As well known, the class of planes corresponds canonically to the class of geometric lattices of length $3$ (cf. \cite[Section 2]{geometric_lattices}). For our purposes the perspective of geometric lattices is preferable (see Convention~\ref{remark_choice_signature} and the proofs of Theorems~\ref{Theorem:Paolini:1} and~\ref{Theorem:Paolini:2}), and thus in the next two remarks we make explicit this correspondence.

	\begin{remark}\label{corr1} Let $P$ be a plane as in Definition~\ref{def_plane}. First of all, add to $P$ a largest element $1$ and a smallest element $0$, and let $P^+ = P \cup \{ 0, 1 \}$. Then, for every pair of points $a$ and $b$ from $P$ let $a \vee b$ denotes the unique line they determine. For every pair of lines $\ell_1$ and $\ell_2$ from $P$ let $\ell_1 \wedge \ell_2$ be $0$ if the two lines are parallel, and let it be the unique point in their intersection otherwise. Then $(P^+, 0, 1, \vee, \wedge)$ is a geometric lattice of length $3$ (cf. \cite[Section 2]{geometric_lattices}).
\end{remark}

	\begin{remark}\label{corr2} Let $(P, 0, 1, \vee, \wedge)$ be a geometric lattice of length $3$ (cf.~\cite[Section~2]{geometric_lattices}), and $P^- = P - \{0, 1 \}$. Let $A$ be the set of atoms of $P$ and let $B$ be the set of co-atoms of $P$, and for $a \in A$ and $b \in B$, let $a \mathbin{E} b$ if $a \vee b = b$. Then $(P^-, A, B, E)$ is a plane, where $A$ denotes the set of points of the plane, $B$ denotes the set of lines of the plane and $E$ is the incidence relation between points and lines.
\end{remark}

	\begin{remark}\label{remark_complete_embed} Let $P_1$ and $P_2$ be planes, and $P_1^+$ and $P_2^+$ be the associated geometric lattices. Then $P_1$ is a complete subplane of $P_2$ iff $P_1^+$ is a sublattice of $P_2^+$.
\end{remark}

	\begin{convention}\label{remark_choice_signature} For the rest of the paper, formally, by a plane we will mean a geometric lattice of length $3$ considered with respect to the signature $L = \{ 0, 1,  \vee, \wedge \}$. Consequently, an embedding of planes will mean an embedding of geometric lattices.
\end{convention}

\section{Proofs of Main Theorems}

	First of all we stress:
	
	\begin{quote} {\bf We invite the reader to keep in mind Convention~\ref{remark_choice_signature}.}
\end{quote}
	
	\begin{notation}\label{notation_classes}
	\begin{enumerate-(1)}
	\item We denote by \(X_\mathrm{Gr}\) the standard Borel space of countable graphs.
	\item We denote by \(X_\mathrm{pl}\) the standard Borel space of simple planes\footnote{Clearly the class of countable simple planes is a first-order class.} (cf. Definition~\ref{def_plane}), and by $\embeds_\mathrm{pl}$ and $\cong_\mathrm{pl}$ the relation of embeddability and isomorphism on \(X_\mathrm{pl}\), respectively.
	\item We denote by \(X_\mathrm{ppl}\) the standard Borel space of countable non-Desarguesian projective planes\footnote{The class of countable non-Desarguesian projective planes is first-order (see e.g. \cite[Definition 5.1.1]{steven}).}, and by $\embeds_\mathrm{ppl}$ and $\cong_\mathrm{ppl}$ the relation of embeddability and isomorphism on \(X_\mathrm{ppl}\), respectively.
	\end{enumerate-(1)}
\end{notation}

\begin{theorem}\label{Theorem:Paolini:1} For every \(\Gamma \in X_\mathrm{Gr}\) let \(P_{\Gamma}\) be defined as in \cite[Section 3]{Pao_proj}.
The map \(X_\mathrm{Gr}\to X_\mathrm{pl} \colon \Gamma \mapsto P_{\Gamma}$ (cf. Notation~\ref{notation_classes}) is Borel and:
\begin{enumerate-(1)}
\item \label{item:1} \(\Gamma_{1}\cong \Gamma_{2}\) if and only if \(P_{\Gamma_{1}}\cong P_{\Gamma_{2}}\);
\item \label{item:2} \(\Aut{\Gamma}\cong \Aut{P_{\Gamma}}\);
\item \label{item:3} \(\Gamma_{1}\sqsubseteq \Gamma_{2}\) if and only if \(P_{\Gamma_{1}}\sqsubseteq P_{\Gamma_{2}}\).
\end{enumerate-(1)}
\end{theorem}

	\begin{proof} Items~\ref{item:1} and~\ref{item:2} are proved in \cite{Pao_proj}. Concerning~\ref{item:3}, argue as in the end of the proof of the main theorem (where it is proved that $\Gamma \mapsto P_{\Gamma}$ is isomorphism-invariant). Notice that the choice of signature (and thus of embedding) is crucial for the argument to go through, since we need that intersection of lines are preserved by the embedding in order to use the $(\star_1)$ of the proof in the way we use it there.
\end{proof}

\begin{theorem}\label{Theorem:Paolini:2} For every \(\Gamma \in X_\mathrm{Gr}\) let \(P^*_{\Gamma}\) be defined as in \cite[Section 4]{Pao_proj}.
The map \(X_\mathrm{Gr}\to X_\mathrm{ppl} \colon \Gamma \mapsto P^*_{\Gamma}$ (cf. Notation~\ref{notation_classes}) is Borel and:
\begin{enumerate-(1)}
\item \label{item:1} \(\Gamma_{1}\cong \Gamma_{2}\) if and only if \(P_{\Gamma_{1}}\cong P_{\Gamma_{2}}\);
\item \label{item:2} \(\Aut{\Gamma}\cong \Aut{P_{\Gamma}}\);
\item \label{item:3} \(\Gamma_{1}\sqsubseteq \Gamma_{2}\) if and only if \(P_{\Gamma_{1}}\sqsubseteq P_{\Gamma_{2}}\).
\end{enumerate-(1)}
\end{theorem}

	\begin{proof} Items~\ref{item:1} and~\ref{item:2} are proved in \cite{Pao_proj}. Concerning~\ref{item:3}, it follows from Remark~\ref{remark_complete_embed}, Facts~\ref{fact_proj1} and~\ref{fact_proj2}, \cite[$(\star_1)$ of Proof of Theorem 3]{Pao_proj} and Theorem~\ref{Theorem:Paolini:1}(3). Notice that also in this case the choice of signature (and thus of embedding) is crucial.
\end{proof}

\begin{proof}[Proof of Theorem~\ref{th_inva_un1}]
Consider the restriction of the map \(\Gamma \mapsto P_{\Gamma}\) on \(\XX\) from Fact~\ref{proposition : FriMot}.
By items~\ref{item:1} and~\ref{item:3} of Theorem~\ref{Theorem:Paolini:1} the map \(\Gamma\mapsto P_{\Gamma}\) simultaneously reduces \(\embeds_{\XX}\) to \(\embeds_\mathrm{pl}\) and \(\cong_{\XX}\) to \(\cong_\mathrm{pl}\).
Condition~\ref{conditionII : CMMR13} of Fact~\ref{theorem : suff condition}
follows by Theorem~\ref{Theorem:Paolini:1}\ref{item:2} and Remark~\ref{remark: simplifiedsuffcondition}.
The statement now follows from Fact~\ref{theorem : suff condition}.
\end{proof}

	\begin{proof}[Proof of Corollary~\ref{first_cor}]
	The statement follows from Theorem~\ref{th_inva_un1} and Fact~\ref{fact: elmentary subclasses}.
\end{proof}

	\begin{proof}[Proof of Theorem~\ref{th_inva_un2}] Argue as in the proof of Theorem~\ref{th_inva_un1} using Theorem~\ref{Theorem:Paolini:2}.
\end{proof}

	\begin{proof}[Proof of Corollary~\ref{second_cor}] The statement follows from Theorem~\ref{th_inva_un2} and Fact~\ref{fact: elmentary subclasses}.
\end{proof}

\section{SPB reductions}

In this section we will denote by \(G\) a Polish group and by \(X\) and \(Y\) two standard Borel spaces.
If \(a\colon G\times X\to X\) is a Borel action of \(G\) on \(X\), we shall denote by \(E_{a}\) the orbit equivalence relation arising from \(a\) (i.e., \(x\mathbin{E_{a}}y \iff \exists g\in G\ (\mathop a(g,x)=y)\)).
The stabilizer of any point \(x\in X\) is the subgroup \(G_{x} \coloneqq \setm{g\in G}{\mathop a(g,x)=x}\).

\begin{definition}\label{def : SPB reduction}
Let  \(a\colon G\times X\to X\), \(b\colon G\times Y\to Y\) be two Borel actions so that \(X\) and \(Y\) are standard Borel \(G\)-spaces.
We say that
\( E_{a}\) \emph{SPB reduces} to \(E_{b}\) (in symbols, \(E_{a}\leq_\mathrm{SPB}E_{b}\)) if there is a Borel map
\(f\colon X\to Y\) witnessing that
 \(E_{a} \leq_\mathrm{B} E_{b}\) and such that
\begin{equation} \tag{\text{SP}}\label{eq:SP}
\forall x\in X ( G_{x}\cong G_{f(x)})\,.
\end{equation}

\end{definition}

We stress the following:
\begin{remark}
When \(G=S_{\infty}\), condition \eqref{eq:SP} is equivalent to saying that \[\forall x\in X(\Aut{x}\cong \Aut{f(x)})\, .\]
\end{remark}

Items~\ref{item:1}--\ref{item:2} of both Theorem~\ref{Theorem:Paolini:1} and Theorem~\ref{Theorem:Paolini:2} can be briefly reformulated as follows.

\begin{theorem}[\cite{Pao_proj}]\label{th_condensed} The following SPB reductions hold:
\begin{itemizenew}
\item \({\cong_\mathrm{Gr}}\leq_\mathrm{SPB} {\cong_\mathrm{pl}}\);
\item \({\cong_\mathrm{Gr}}\leq_\mathrm{SPB} {\cong_\mathrm{ppl}}\).
\end{itemizenew}\end{theorem}

We highlight the following fact which follows directly from Fact~\ref{theorem : suff condition} and Remark~\ref{remark: simplifiedsuffcondition}, and exhibits how Theorem~\ref{th_condensed} can be used to prove Theorem~\ref{th_inva_un1} and Theorem~\ref{th_inva_un2}.
\begin{fact}\label{Fact:reformulated} Let \(X\) be a standard Borel space of countable structures. Then the relation \(\embeds_{X}\) is an invariantly universal quasi-order
provided that there is a Borel map \(f \colon \mathbb{X}\rightarrow X\) such that:
 \begin{enumerate-(i)}
 \item \label{i : CMMR13}\(f\colon {\embeds_\mathbb{X}}\leq_\mathrm{B}{\embeds_{X}}\);
 \item \label{ii : CMMR13}\(f\colon {\cong_\mathbb{X}}\leq_\mathrm{SPB}{\cong_{X}}\).
  \end{enumerate-(i)}
\end{fact}

Some examples of SPB reductions directly follow from the existence of full embeddings between categories.
In category theory there has been quite a lot of work concerning the complexity of different categories by means of (categorical) embeddings\footnote{The categorical notion of embedding should not be confused with the one of embedding between structures.}. Several classical examples of categorical embedding typically concern categories whose objects are algebraic structures of a fixed type, and whose morphisms are the respective homomorphisms (or embedding) between those structures. A comprehensive reference for this kind of results is the book~\cite{PT}. One of the strongest notion of (categorical) embedding that has been considered in the literature is the one of \emph{full embedding}, an injective functor which further induces a bijection between the morphisms in the domain category and the morphisms in the target category.

\begin{definition}
If \(\mathcal{C}\) and \(\mathcal{D}\) are categories, a \emph{full embedding} \(F\) from \(\mathcal{C}\) into \(\mathcal{D}\) is a functor \(F\colon \mathcal{C}\to \mathcal{D}\) such that
\begin{itemizenew}
\item \(F\) is injective on the objects of \(\mathcal{C}\);
\item for every \(a,b\) the map \({\mathrm{Hom}_{\mathcal C}(a,b)}\to {\mathrm{Hom}_{\mathcal D}(F(a),F(b))}\colon f\mapsto F(f)\) is a bijection.
\end{itemizenew}
\end{definition}

An example of full embedding is given by the constructions of the second author, that we previously mentioned in the statements of Theorem~\ref{Theorem:Paolini:1} and Theorem~\ref{Theorem:Paolini:2}. E.g., the map \(\Gamma\mapsto P^{*}_{\Gamma}\) can be redefined for the category of all graphs, regardless of their cardinality (and in fact this is the setting of \cite{Pao_proj}), to prove the following:

\begin{theorem}[essentially \cite{Pao_proj}]
There exists a full embedding from the category of graphs together with graph embeddings into the category of non-Desarguesian projective planes together with planes embeddings (recall Convention~\ref{remark_choice_signature}).
\end{theorem}

Our interest in full embeddings is easily explained. First, notice that any \(\mathfrak{L}_{\omega_{1} \omega}\)-class \(\mathcal{C}\) can be regarded as a category --- the morphisms of \(\mathcal{C}\) are the usual embeddings between the structures which \(\mathcal{C}\) is formed by. Then, next proposition explains how certain full embedding induce a Borel reduction.

\begin{proposition}\label{full embedding->SPB}
Let \(\mathcal{C}\) and \(\mathcal{D}\) be two \(\mathfrak{L}_{\omega_{1}\omega}\)-classes so that we can consider the corresponding standard Borel spaces \(X_\mathcal{C}\) and \(X_\mathcal{D}\).
Suppose that \(F\) is a full embedding from \(\mathcal{C}\) into \(\mathcal{D}\) such that
\begin{enumerate-(i)}
\item \label{item(i)}  \(F\) preserves countability;
\item \label{item(ii)} \(F\) can be realized as a Borel function from \(X_\mathcal{C}\) to \(X_\mathcal{D}\); i.e.,  there is a Borel function \(f\colon X_{\mathcal{C}}\to X_{\mathcal{D}}\) such that
 for every \(x\in X_{\mathcal{D}}\), \(f(x)\cong F(x)\).
\end{enumerate-(i)}
Then, the isomorphism relation \(\cong_\mathcal{C}\) SPB reduces to \(\cong_\mathcal{D}\).
\end{proposition}
\begin{proof}
Since \(F\) is full, for every \(x,y\), the sets of isomorphisms between \(x\) and \(y\) and their images, respectively denoted by \(\mathrm{Iso}(x,y)\) and \(\mathrm{Iso}(F(x),F(y))\), are isomorphic via the map
\[
\mathrm{Iso}(x,y)\to\mathrm{Iso}(F(x),F(y)) \colon h\mapsto F(h).
\]
In particular, for every \(x\in X_{\mathcal{C}}\), the map:
  \[\mathrm{Aut}(x)\to\mathrm{Aut}(F(x)) \colon h\mapsto F(h).\]
   is a bijection, indeed it is a group isomorphism.
Now let \(f\colon X_{\mathcal{C}}\to X_{\mathcal{D}}\) be a Borel function as in \ref{item(ii)}. Since every \(F(x)\) is isomorphic to \(f(x)\), we have that for every \(x\in X_{\mathcal{C}}\), \[\mathrm{Aut}(x)\cong\mathrm{Aut}(F(x)).\] 
 \end{proof}
 
We add one more comment to Proposition~\ref{full embedding->SPB}. Following the approach of~\cite{Lup}, we can regard the subcategories of \(\mathcal{C}\) and \(\mathcal{D}\) formed by \(X_\mathcal{C}\) and \(X_\mathcal{D}\), respectively, together the isomorphism maps as analytic groupoids. The SPB reduction we get is in particular a functorial reduction (see~\cite[Definition~2.8.1]{Lup}).

The following full embeddings between categories are well-known in the literature. 
When not specified, we consider categories with respect to embeddings homomrphisms as morphisms.

\begin{fact}
There is a full embedding from the category of graphs into any of the following categories.
\begin{itemizenew}
\item the category of partial orders \(\mathrm{PO} \) (\cite[Chapter IV, 5.6]{PT});
\item the category of semigroups \(\mathrm{Smg} \) (\cite[Chapter V, 2.9]{PT});
\item the category of unital rings \(\mathrm{Rng_{1}}\) (\cite[Section 3]{FriSch}).
\end{itemizenew}
\end{fact}

 One can check that for any of the aforementioned categorical embeddings, items \ref{item(i)}--\ref{item(ii)} of Proposition~\ref{full embedding->SPB} are satisfied, thus we obtain the following.

\begin{proposition}
The isomorphism relation between countable graphs \(\cong_\mathrm{Gr}\) SPB reduces to any of the following isomorphism relation
\begin{itemizenew}
\item the isomorphism relation between countable partial orders \(\cong_{\mathrm{PO}} \);
\item the isomorphism relation between countable semigroups \(\cong_{\mathrm{Smg}} \) ;
\item the isomorphism relation between countable unital rings \(\cong_{\mathrm{Rng_{1}}}\).
\end{itemizenew}
\end{proposition}

We conclude this section with a few more thoughts about SPB reductions.
If the \(\mathfrak{L}_{\omega_{1}\omega}\)-elementary classes \(X\) and \(Y\) are Borel complete, then the isomorphism relations
\({\cong_{X}}\) and \({\cong_{Y}}\) are necessarily Borel bi-reducible, but they need not be SPB bi-reducible. E.g., the isomorphism relation between countable graphs \(\cong_\mathrm{Gr}\) does not SPB reduce to isomorphism between countable groups \(\cong_\mathrm{Gp}\), because every infinite countable group have nontrivial automorphisms.

Let \(\cong_\mathrm{Tr}\) be the isomorphism relation between countable trees (i.e., connected acyclic graphs) and \(\cong_\mathrm{LO}\) be the isomorphism relation between countable linear orders. Although \(\cong_\mathrm{Tr}\) and \(\cong_\mathrm{LO}\) have been known to be Borel complete, they are not equivalent to \(\cong_\mathrm{Gr}\) up to faithful Borel reducibility (cf.~\cite[Theorem~4.5]{Gao01}). It is then natural to ask the following questions.

\begin{question}
Does \({\cong_\mathrm{Gr}}\leq_\mathrm{SPB}{\cong_\mathrm{Tr}}\)?
\end{question}

\begin{question}
Does \({\cong_\mathrm{Gr}}\leq_\mathrm{SPB}{\cong_\mathrm{LO}}\)?
\end{question}


\begin{thebibliography}{10}

\bibitem{BroCalMil}
Andrew D. Brooke-Taylor, Filippo Calderoni, and Sheila K. Miller.
\newblock {\em Invariant Universality for Quandles and Fields}.
\newblock Submitted, available on the ArXiv.

\bibitem{CalMot}
Filippo Calderoni and Luca Motto Ros.
\newblock {\em Universality of Group Embeddability}.
\newblock Proc. Amer. Math. Soc., to appear.

\bibitem{CamMarMot}
Riccardo Camerlo, Alberto Marcone, and Luca Motto Ros.
\newblock {\em Invariantly Universal Analytic Quasi-Orders}.
\newblock Trans. Amer. Math. Soc. {\bf 365} (2013), no. 4, 1901--1931.

\bibitem{FriMot}
Sy-David Friedman and Luca Motto Ros.
\newblock {\em Analytic Equivalence Relations and Bi-Embeddability}.
\newblock J. Symbolic Logic {\bf 76} (2011), no. 1, 243--266.

\bibitem{FriSch}
Ervin Fried and Ji\v r\'i Sichler.
\newblock {\em Homomorphisms of commutative rings with unit element}.
\newblock Pacific J. Math. {\bf 45} (1973), 485--491.

\bibitem{Gao01}
Su Gao. 
\newblock {\em Some dichotomy theorems for isomorphism relations of countable models}.
\newblock  J. Symbolic Logic {\bf 66} (2001), no. 2, 902--922.

\bibitem{hall_proj}
Marshall Hall.
\newblock {\em Projective Planes}.
\newblock Trans. Amer. Math. Soc. {\bf 54} (1943), 229--277.

\bibitem{Kec}
Alexander~S. Kechris.
\newblock {\em Classical descriptive set theory}, volume 156 of {\em Graduate
  Texts in Mathematics}.
\newblock Springer-Verlag, New York, 1995.

\bibitem{LouRos}
Alain Louveau and Christian Rosendal.
\newblock {\em Complete analytic equivalence relations}.
\newblock Trans. Amer. Math. Soc. {\bf 357} (1943), no. 12, 4839--4866.

\bibitem{piper}
Daniel R. Hughes and Fred C. Piper.
\newblock {\em Projective Planes}.
\newblock Graduate Texts in Mathematics, Vol. 6. Springer-Verlag, New York-Berlin, 1973.

\bibitem{geometric_lattices}
Tapani Hyttinen and Gianluca Paolini.
\newblock {\em Beyond Abstract Elementary Classes: On The Model Theory of Geometric Lattices}.
\newblock Ann. Pure Appl. Logic {\bf 169} (2018), no. 2, 117--145.

\bibitem{Lup}
Martino Lupini.
\newblock {\em Polish groupoids and functorial complexity}.
\newblock Trans. Amer. Math. Soc. {\bf 369} (2017), no. 9, 6683--6723.

\bibitem{Pao_proj}
Gianluca Paolini.
\newblock {\em The Class of Countable non-Desarguesian Projective Planes is Borel Complete}.
\newblock Proc. Amer. Math. Soc, to appear.

\bibitem{PT}
   Ale\v s Pultr and V\v era Trnkov\'a
   \newblock {\em Combinatorial, algebraic and topological representations of
   groups, semigroups and categories}, North-Holland, 1980.

\bibitem{steven}
Fredrick W. Stevenson.
\newblock {\em Projective Planes}.
\newblock W. H. Freeman and Co., San Francisco, Calif., 1972. 

\end{thebibliography}
\end{document}